\documentclass[12pt]{amsart}
\usepackage{amsfonts}
\usepackage{amsmath}
\usepackage{amssymb}
\usepackage{graphicx}
\usepackage{amscd}
\usepackage{xcolor}
\usepackage{hyperref}
\DeclareMathOperator\arctanh{arctanh}
\newtheorem{theorem}{Theorem}

\newtheorem{corollary}[theorem]{Corollary}

\newtheorem{example}[theorem]{Example}

\newtheorem{proposition}[theorem]{Proposition}

\newcommand{\ch}{\cosh}
\newcommand{\sh}{\sinh}

\begin{document}
	\author{G. Polychrou}
	\title[Solutions of the Sinh-Gordon and sine-Gordon equations and applications]{Solutions of the Sinh-Gordon and sine-Gordon equations and applications}
	\maketitle
	\begin{abstract}
		We study the elliptic sinh-Gordon and sine-Gordon equations on the real plane and we introduce new families of solutions. We use a B{\"a}cklund transformation that connects the elliptic versions of sine-Gordon and sinh-Gordon equations. As an application, we construct new harmonic maps between surfaces, when the target is  of constant curvature $-1$.
	\end{abstract}
\section{Introduction and Statement of the Results}\label{Introduction}
Our aim is to find new solutions of the elliptic sinh-Gordon equation 
\begin{equation}\label{sinh-Gordon eqn}
	\Delta w=k\sinh(2w) , k \in \mathbb{R},
\end{equation}
and also to find new solutions of the elliptic sine-Gordon equation 
\begin{equation}\label{sine-Gordon eq}
	\Delta \theta = - l\sin(2\theta), l \in \mathbb{R},
\end{equation} 
and apply them to the harmonic map problem. 
For simplicity we shall only consider the case $k = 2$ and $l = -2$.

Firstly, we consider solutions $w$ of the sinh-Gordon equation (\ref{sinh-Gordon eqn}) that are of the form
\begin{equation}\label{tan}
	\sh w(x,y) = \tan(A(x) + B(y)),
\end{equation} 
for some real functions $A(x)$ and $B(y)$, which we calculate explicity. Generally, these functions turn out to be Jacobi elliptic, \cite{Armitage}.  Next we apply these solutions to the harmonic map problem, as in \cite{FotDask} and \cite{PPFD2021}. 

Next, we describe solutions of (\ref{sine-Gordon eq}) that are either of the form
\begin{equation}\label{FanGsine}
	\tan(\frac{\theta(x,y)}{2}) = F(x)G(y)	,
\end{equation}
or of the form
\begin{equation}\label{tanh}
	\sin(\theta(x,y)) = \tanh(C(x) + D(y)).
\end{equation}
 
In \cite{FotDask}, the authors introduce a B{\"a}cklund transformation between the solutions of the elliptic sinh-Gordon equation and the elliptic sine-Gordon. There is a \textit{B{\"a}cklund transformation} that connects solutions of equation (\ref{sinh-Gordon eqn}) and (\ref{sine-Gordon eq});
\begin{align}
	\partial_x w-\partial_y \theta&=-2\sinh w \sin \theta \label{Back1intro},\\
	\partial_y w+\partial_x\theta&=-2\cosh w \cos \theta\label{Back2intro}.
\end{align}
Therefore, given a solution $w$ of the sinh-Gordon equation, we can construct the corresponding solution $\theta$ of the sine-Gorodn equation and vice versa.

 Using the B{\"a}cklund transformation we construct new families of solutions of (\ref{sinh-Gordon eqn}). Finally, we apply these solutions in each of the above cases to find new solutions to the harmonic map problem, as in \cite{FotDask}.
 
 The study of the sinh-Gordon and sine-Gordon has been extensive (see \cite{FokPell,Hwang}). Sinh-Gordon equation has many applications in Geometric Analysis, for example in CMC surfaces 
 (see \cite{Joaquin,K}). Futhermore, the sinh-Gordon equation apply in the work on the Wente torus (see \cite{Abresch,Wente}). 
 
 There are many examples of harmonic diffeomorphisms, see for example \cite{LiTam,Minsky,Schoen,Wolf1} and their references therein. The study of harmonic maps between Riemann surfaces with constant mean curvature plays a vital role in the theory of harmonic maps, see for example \cite{FotDask,J,PPFD2021}.
 
The outline of this paper is as follows. In section 2, we review some known results. In section 3 we use solutions of the form (\ref{tan}) and we construct the corresponding harmonic map. In the last section, we use solutions of (\ref{sine-Gordon eq}) of the form (\ref{FanGsine}) and (\ref{tanh}) to construct solutions of the sinh-Gordon equation (\ref{sinh-Gordon eqn}) and the harmonic map corresponding to them.

\section{Preliminaries}\label{sec:Preliminaries}

In this section we present some known results.
Let $u : M \rightarrow N$ be a map between
Riemann surfaces $(M, g)$, $(N, h)$. We can  write  the map locally as
$u = u(z, \bar{z}) = R + iS$, where $z=x+iy$. From now on, we use the standard notation:
\[
\partial_z=\frac{1}{2}(\partial_x-i\partial_y), \quad \partial_{\bar{z}}=\frac{1}{2}(\partial_x+i\partial_y).
\]
Also, we introduce  isothermal coordinates $(x, y)$ on $M$ such that
\[g =  e^{f(z,\bar{z})}|dz|^2,\]
where $z = x + iy$, and an isothermal coordinate system $(R, S)$ on
$N$ such that
\begin{equation*}\label{targetconf}h =  e^{F(u,\bar{u})}|du|^2,
\end{equation*}
where $u = R + iS$. In isothermal coordinates,  a map between two surfaces is harmonic if and only if it satisfies
\begin{equation*}
	\partial_{z\bar{z}}u + \partial_{u}F(u,\bar{u}) \partial_{z}u\partial_{\bar{z}}u=0;
\end{equation*}
see \cite{J}. Note that this equation depends only on the conformal structure of the target. 
We define the Hopf differential of $u$ by
\begin{equation*}\label{eq:Hopf_theorem}
	\Lambda(z) dz^2=\left( e^{F(u,\bar{u})}\partial_{z} u \partial_{z} \bar{u}\right)dz^2.
\end{equation*}
The following proposition is well-known.
\begin{proposition}[\cite{J}]\label{prop:Hopf}
	A necessary and sufficient condition for  $u$ to be a harmonic map, is that the Hopf differential of $u$ is holomorphic, i.e.
	\begin{equation}\label{eq:Hopf}
		e^{F(u,\bar{u})} \partial_z u    \partial_z\bar{u}= e^{-\mu (z)},
	\end{equation}
	where $\mu (z)$ is a holomorphic function.
\end{proposition}
On the domain surface $M$, we can choose a specific coordinate system in order to considerably facilitate the calculations. This \textit{specific}  system is defined by the conformal transformation
\begin{equation*}\label{eq:specific}
	Z= \int e^{-\mu(z)/2} \, dz.
\end{equation*}   
In this specific system, the above equations simplify by considering $\mu(z)=0$.
In particular, the harmonic map condition (\ref{eq:Hopf}) becomes \begin{equation}\label{Hopfnew}
	e^{F(u, \bar{u})}\partial_{z} u\partial_{z}\bar{u}=1.
\end{equation} 
We shall say that the harmonic map corresponds to a solution $w$ of (\ref{sinh-Gordon eqn}) if
\[
\frac{\partial_{\bar{z}}u}{\partial_{z}u} = e^{-2w}.
\] 
We now recall a main result of \cite{FotDask}.
\begin{proposition}[\cite{FotDask}]\label{char}
	 Let $\phi = constant$ be the solution of the characteristics 
	\cite{PR},
	\begin{equation}\label{characteristic}
		\frac{dy}{dx} = i \coth(w(x,y)).
	\end{equation}
	Then, a harmonic map that corresponds to $w$ is of the form
	\begin{equation}
		u(x,y) = Re\Phi(x,y) + iIm\Phi(x,y).
	\end{equation}
\end{proposition}
Consider now $w$ a solution of (\ref{sinh-Gordon eqn}) and $\theta$ a solution of (\ref{sine-Gordon eq}). Set 
\begin{align*}I_1&=I_1(x)=\int_0^x\cosh w(t,0)\sin\theta(t,0)dt, \\
	I_2&=I_2(x,y)=\int_0^y\sinh w(x,s)\cos\theta(x,s)ds\\
	I_3&=I_3(x)=\int_0^x e^{2 I_1 (t)}\cosh w(t,0)\cos\theta(t,0)dt,\\
	I_4&=I_4(x,y)=e^{2 I_1 (x)}\int_0^y e^{2 I_2 (x,s)}\sinh w(x,s)\sin\theta(x,s)ds.
\end{align*}
\begin{proposition}(\cite{PPFD2021})\label{PPFD}
	Define the function $S$ by 
	\begin{align*}
		S(x,y)=S(0,0)e^{2(I_1+I_2)}, 
	\end{align*}
	and the function $R$ by
	\begin{align*}
		R(x,y)=R(0,0)+2S(0,0)(I_3-I_4). 
	\end{align*}
	Then, 
	\[
	u(x,y)=R(x,y)+iS(x,y)
	\]
	is a harmonic map that corresponds to $w$.
\end{proposition}
\section{The sinh-Gordon equation}\label{sec:The sinh-Gordon equation}
In this section, we study a new family of solutions of the sinh-Gordon equation (\ref{sinh-Gordon eqn}) that are of the form 
\begin{equation}\label{newfam}
\sinh\omega(x,y) = \tan(A(x) + B(y)).
\end{equation}
We shall find a family of such solutions and then we shall give an example of a corresponding harmonic map.
\begin{theorem}\label{aandb}
	If $w$ is a solution of the sinh-Gordon equation (\ref{sinh-Gordon eqn}) of the form
	$$
	\sinh\omega(x,y) = \tan(A(x) + B(y)) = \tan(A_0 + B_0 + \int_{0}^{x}a(t)dt + \int_{0}^{y}b(s)ds),
	$$
	then the functions $a(x)$ and $b(y)$ satisfy the differential equations
	\begin{equation}\label{aeq}
		(a'(x))^{2} = - (a(x))^{4} + c_1(a(x))^{2} + c_2,
	\end{equation}
	\begin{equation}\label{beq}
		(b'(y))^{2} = -(b(y))^{4} + (8 - c_1)(b(y))^{2} + c_3.
	\end{equation}
where $a(x) = A'(x)$, $b(y) = B'(y)$ and $c_1,c_2$ and $c_3$ are constants such that $4c_1 = 16 + c_3 - c_2$. 
\end{theorem}
\begin{proof}
	Differentiating $w(x,y)$ and replacing into (\ref{sinh-Gordon eqn}) we have
	\[
	(A''(x) + B''(y))\cosh w + ((A'(x))^{2} + (B'(y))^{2})\sinh w \cosh w
	= 2\sinh 2w.
	\]
	Using the fact that $\sinh\omega(x,y) = \tan(A(x) + B(y))$ and dividing with $\cosh w$ we get
	\[
	\tan(A(x) + B(y)) = \frac{A''(x) + B''(y)}{4 - (A'(x))^{2} - (B'(y))^{2}}.
	\]
	Differentiating with respect to $x$ 
	and using the identity $1/\cos^{2}\phi = 1 + \tan^{2}\phi$ we obtain
	\begin{equation}\label{AandB}
		A'(x)(4 - (A'(x))^{2} - (B'(y))^{2})^{2} + A'(x)(A''(x) + B''(y))^{2} 
	\end{equation}
	\[
	= A'''(x)(4 - (A'(x))^{2} - (B'(y))^{2}) + 2A'(x)A''(x)(A''(x) + B''(y)).
	\]
	Differentiating with respect to $y$ we have
	\[
	- 4A'(x)B'(y)B''(y)(4 - (A'(x))^{2} - (B'(y))^{2}) + 2A'(x)B'''(y)(A''(x) + B''(y))
	\]
	\[
	= - 2B'(y)B''(y)A'''(x) + 2A'(x)A''(x)B'''(y).
	\]
	Therefore we get that
	\[
	\frac{A'''(x)}{A'(x)} + 2(A'(x))^{2} = 8 - \frac{B'''(y)}{B'(y)} - 2(B'(y))^{2} = c_{1}.
	\]
	So we derive the equations
	\begin{equation}\label{DerA}
		A'''(x) + 2(A'(x))^{3} = c_{1}A'(x),
	\end{equation}
	\begin{equation}\label{DevB}
		B'''(y) + 2(B'(y))^{3} = (8 - c_{1})B'(y).
	\end{equation}
	Starting with equation (\ref{DerA}), we have
	\[
	2A''(x)A'''(x) + 4(A'(x))^{3}A''(x) = c_{1}2A'(x)A''(x)
	\]
	\[
	\Rightarrow
	(A''(x))^{2} + (A'(x))^{4} = c_{1}(A'(x))^{2} + c_2.
	\]
	By letting $A(x) = A_0 + \int_{0}^{x}a(t)dt$, we obtain
	\begin{equation}\label{eqa}
		(a'(x))^{2} = - (a(x))^{4} + c_{1}(a(x))^{2} + c_2.
	\end{equation}
	In a similar manner, we get
	\begin{equation}\label{eqb}
		(b'(y))^{2} = -(b(y))^{4} + (8 - c_{1})(b(y))^{2} + c_3,
	\end{equation}
	where $B(y) = B_0 + \int_{0}^{y}b(s)ds$. Finally, using (\ref{AandB}) we derive the connection
	$4c_1 = c_3 - c_2 + 16$. 
\end{proof}
\begin{corollary}
	By using the initial conditions $A'(0) = a(0) = B'(0) = b(0) = 0$, we have $c_2 = a'(0)^2 = 16\alpha^2$, 
	$c_3 = b'^(0)^2 = 16\beta^2$ and $c_1 = 4(1 - \alpha^2 + \beta^2)$ and the equations (\ref{aeq}) and (\ref{beq}) turn into
	\begin{equation}\label{aspecial}
		(a'(x))^{2} = - (a(x))^{4} + 4(1 - \alpha^{2} + \beta^{2})(a(x))^{2} + 16\alpha^{2},
	\end{equation}
	\begin{equation}\label{bspecial}
		(b'(y))^{2} = -(b(y))^{4} + 4(1 + \alpha^{2} - \beta^{2})(b(y))^{2} + 16\beta^{2}.
	\end{equation}
\end{corollary} 
\par  In \cite{FotDask}, the authors prove that using solutions of the elliptic sinh-Gordon equation they are able to construct harmonic maps to a hyperbolic surface. Similarly, we shall use Proposition \ref{char}. Let us consider in Theorem \ref{aandb} $c_1 = 4, c_2 =  c_3 = 0$. Then, we get the equations
\begin{equation}
	(a'(x))^2 = - (a(x))^4 + 4(a(x))^2, (b'(y))^2 = - (b(y))^4 + 4(b(y))^2.
\end{equation}
One can easily solve these equations and so
$a(x) = 2 sech(2x)$ and $b(y) = 2 sech(2y)$. We have to emphasize that we consider these special constants, in order to avoid the elliptic functions and make the calculations far easier. So the solution of the new family presented in Section 3, turns out to be
\begin{equation}\label{specialsol1}
	\sh w(x,y) = \frac{\sh(2x) + \sh(2y)}{1 - \sh(2x)\sh(2y)}.
\end{equation} 
We shall now provide the corresponding harmonic map.
\begin{example}
Using Proposition (\ref{char}), we have to solve the following equation in order to find that the corresponding harmonic map:
\begin{equation}
	\frac{dy}{dx} = i \coth w(x,y) = i \frac{\ch(2x)\ch(2y)}{\sh(2x) + \sh(2y)}.
\end{equation}
By lengthy calculations we find that a corresponding harmonic map is
\begin{equation}\label{harmonic1}
	u(x,y) = R(x,y) + iS(x,y),
\end{equation}
where
\begin{equation}\label{R1}
	R(x,y) = sech(2y) - \sh(2x)\tanh(2y)
\end{equation}
and 
\begin{equation}\label{S1}
	S(x,y) = \sh(2x)sech(2y) + \tanh(2y) - 2y.
\end{equation}
An implicit formula for the metric on the target of curvature $-1$ is
\begin{equation}\label{targetmetric}
	e^{F(u,\bar{u})}dud\bar{u} = 
	4\frac{\ch^2(2x)\ch^2(2y)dx^2 + (\sh(2x) + \sh(2y))^2dy^2}
	{(1 - \sh(2x)\sh(2y))^2}
\end{equation}
where $x = x(R,S)$ and $y = y(R,S)$.
\end{example}
\section{Solutions via the elliptic sine-Gordon equation}
In this section we present new solutions of the sine-Gordon equation (\ref{sine-Gordon eq}), next we find solutions of the sinh-Gordon equation (\ref{sinh-Gordon eqn}) and finally we construct harmonic maps to a hyperbolic surface.

At first, we focus on solutions of (\ref{sine-Gordon eq}) of the form
\begin{equation}\label{solsine-Gordon1}
	\theta(x,y) = 2\arctan(F(x)G(y)).
\end{equation}
\begin{proposition}\label{solution1}
	If $\theta(x,y)$ is a solution of the sine-Gordon equation (\ref{sine-Gordon eq}) of the form (\ref{solsine-Gordon1}), then the functions $F(x)$ and $G(y)$ satisfy the following equations 
	\begin{equation}\label{F}
		F'(x)^{2} = AF^{4}(x) + BF^{2}(x) + C,
	\end{equation}
	\begin{equation}\label{G}
		G'(y)^{2} = CG^{4}(y) - (4 + B)G^{2}(y) + A,
	\end{equation}
	where $A,B$ and $C$ are arbitrary constants.
\end{proposition}
\begin{proof}
	Set $G(y) = 1/H(y)$. Dividing equation (\ref{sine-Gordon eq1}) by $2F(x)H(y)$ we get
	\begin{equation}\label{FandH}
		\left(\frac{F''(x)}{F(x)} - \frac{H''(y)}{H(y)}\right)(F^{2}(x) + H^{2}(y)) - 2F'(x)^{2} + 2H'(y)^{2} = 4F^{2}(x) - 4H^{2}(y).
	\end{equation}
	Differentiating with respect to $x$ and $y$, we get
	\[
	\frac{1}{F'(x)F(x)}(\frac{F''(x)}{F(x)})' = \frac{1}{H'(y)H(y)}(\frac{H''(y)}{H(y)})' = 4A.
	\]
	Therefore,
	\[
	(\frac{F''(x)}{F(x)})' = 4AF'(x)F(x), (\frac{H''(y)}{H(y)})' = 4AH'(y)H(y).
	\]
	Integrating twice, we obtain 
	\begin{equation}\label{F(A,B,C)}
	F'(x)^{2} = AF^{4}(x) + BF^{2}(x) + C,
	\end{equation}
	\begin{equation}\label{G(A,B,C)}
	H'(y)^{2} = A'H^{4}(y) + B'H^{2}(y) + C'.
	\end{equation}
	Taking into account (\ref{FandH}) and that $G(y) = 1/H(y)$, equations (\ref{F}) and (\ref{G}) follow.
\end{proof}
\par Since we have a solution of the sine-Gordon equation of the form (\ref{solsine-Gordon1}), we shall demonstrate how to compute the corresponding solution of the sinh-Gordon equation using the B{\"a}cklund transformation. To simplify the calculations, we set
\[
Y = Y(y) = \int_{0}^{y}\cos\theta(0,s)ds, \quad
X = X(x,y) = \int_{0}^{x}\sin\theta(t,y)dt.
\]
\begin{proposition}
	If $\theta(x,y)$ is a solution of (\ref{sine-Gordon eq1}) of the form (\ref{solsine-Gordon1}) with initial conditions $F'(0) = 0$ and $w(0,0) = 0$ then 
	\begin{equation}\label{sinh-Godron1}
		\tanh(\frac{w(x,y)}{2}) = 
		\frac{2 - K\tan(Y) + \sqrt{K^2 + 4}\tanh(\frac{\sqrt{K^2 + 4}}{2}X)}
		{\sqrt{K^2 + 4} + (2 - K\tan(Y))\tanh(\frac{\sqrt{K^2 + 4}}{2}X)},
	\end{equation}
	where $K = K(y) = \frac{H'(y)}{H(y)}$.
\end{proposition}
\begin{proof}
	We shall use the B{\"a}cklund transformation,
	\begin{equation}\label{Back1}
		w_{x} - \theta_{y} = - 2\sinh w\sin\theta,
	\end{equation}
	\begin{equation}\label{Back2}
		w_{y} + \theta_{x} = - 2\cosh w\cos\theta.
	\end{equation}
	Letting $x = 0$ in equation (\ref{Back2}), we obtain
	\[
		\frac{\partial w(0,y)}{\partial y} + \frac{\partial \theta(0,y)}{\partial x} = - 2\cosh w(0,y)\cos\theta(0,y),
	\]
	and since $F'(0) = 0$ and $w(0,0) = 0$, we get
	\[
		\frac{\partial w(0,y)}{\partial y} = - 2\cosh w(0,y)\cos\theta(0,y)
	\]
	\begin{equation}\label{w(0,y)1}
		\tanh\frac{w(0,y)}{2} = - \tan(Y(y)).
	\end{equation}
	Using equation (\ref{Back1}), we have
\[
	\frac{\partial w(x,y)}{\partial x} = \sin(\theta(x,y))(\frac{H'(y)}{H(y)} - 2\sinh w(x,y))
\]
and by integrating we get
\[
\arctanh\left(\frac{K\tanh(\frac{w(x,y)}{2}) + 2}{\sqrt{K^2 + 4}}\right)
- \arctanh(\frac{K\tanh(\frac{w(0,y)}{2}) + 2}{\sqrt{K^2 + 4}})
= \frac{\sqrt{K^2 + 4}}{2}X(x,y).
\]
	After several calculations with the help of (\ref{w(0,y)1}) and the use of trigonometric identies, equation (\ref{sinh-Godron1}) follows and the proof is complete.
\end{proof}
\par Since we calculated a family of solutions of the sinh-Gordon equation, we can construct the corresponding harmonic map to a hyperbolic surface.
\begin{example}\label{ex1}
	Consider the case where $A = 1$,$B = -2$ and $C = 1$. Then we can calculate the solutions of equations (\ref{F}) and (\ref{G}) as the constant functions $F(x) = G(y) = 1$. The solution of sine-Gordon equation is again the constant function $\theta(x,y) = \frac{\pi}{2}$. Using the B{\"a}cklund transformation, we calculate the solution of the sinh-Gordon equation as
	\begin{equation}\label{wex1}
		w(x,y) = 2\arctanh(e^{2x}).
	\end{equation}  
As we construct a solution of the sinh-Gordon equation, we can derive a corresponding harmonic map by using Proposition \ref{char}. The harmonic map is
\begin{equation}\label{harex1}
	u(x,y) = y + \frac{\epsilon\sh(2x)}{2}, 
\end{equation}
This example is an one-soliton solution of sinh-Gordon equation. It was presented for first time in \cite{LiTam} as a harmonic map form the upper half plane to itself.
\end{example}
\begin{example}\label{ex2}
	Consider the solution of sine-Gordon equation
	\begin{equation}\label{ThetaEx2}
	\tan(\frac{\theta(x,y)}{2}) = 2y\sec(2x).
	\end{equation}
	Using the B{\"a}cklund transformation, we calculate the solution of the sinh-Gordon equation as
	\begin{equation}\label{WEx2}
	\tanh(\frac{w(x,y)}{2}) = \frac{\cos(y)(\sin(2x) - 2y) + \sin(y)}{\cos(y) + (2y + \sin(2x))\sin(y)}.
	\end{equation}
	By Proposition \ref{char} we find that a harmonic map is $u(x,y) = R(x,y) + iS(x,y)$, where
\begin{equation}\label{REx2}
	R(x,y) = 
	\frac
	{\cos(2y)\cos^2(2x) + 4y(\sin(2x) + \sin(2y) - y\cos(2y))}
	{4y^2 + \cos^2(2x)},
\end{equation}
\begin{equation}\label{SEx2}
	S(x,y) = 2x + 
	\frac
	{4y\cos(2x)\cos(2y) - 2\cos(2x)(\sin(2x) + \sin(2y))}
	{4y^2 + \cos^2(2x)}.
\end{equation}
The metric on the target in implicit form is 
\begin{equation}\label{metric2}
	I = 4 
	\frac
	{(3 + 8y^2 - \cos(4x) + 4\sin(2x)(\sin(2y) - 2y\cos(2y)))^2dx^2}
	{(\cos(2y)(1 - 8y^2 + \cos(4x)) + 8y(\sin(2x) + \sin(2y)))^2} 
\end{equation}
\[+
4\frac{(8y\cos(2y) - 4\sin(2x) + \sin(2y)(8y^2 + \cos(4x) - 3))^2dy^2}
	{(\cos(2y)(1 - 8y^2 + \cos(4x)) + 8y(\sin(2x) + \sin(2y)))^2}
\]
where $x = x(R,S)$ and $y = y(R,S)$.
\end{example}
\par Secondly, we consider solutions of the sine-Gordon equation
\begin{align}\label{sineGordoneq}
	\Delta\theta(x,y) = - 2\sin(2\theta)
\end{align}
of the form
\begin{align}\label{sin(theta)}
	\sin\theta(x,y) = \tanh(C(x) + D(y)).
\end{align}
\begin{proposition}
	If $\theta$ is a solution of the sine-Gordon equation (\ref{sineGordoneq}) of the form
	\[
	\sin\theta(x,y) = \tanh(C(x) + D(y)) = \tanh(C(0) + D(0) 
	+ \int_{0}^{x}c(t)dt + \int_{0}^{y}d(s)ds)	
	\]
	then the functions $c(x)$ and $d(y)$ satisfy the differential equations
	\begin{align}\label{cdiffeq}
		(c'(x))^{2} = (c(x))^{4} + c_{4}(c(x))^{2} + c_{5}
	\end{align}
	\begin{align}\label{ddiffeq}
		(d'(y))^{2} = (d(y))^{4} - (8 + c_{4})(d(y))^{2} + c_{6},
	\end{align}
	where $c(x) = C'(x)$, $d(y) = D'(y)$ and $c_{4},c_{5},c_{6}$ are arbitrary constants such that $16 + 4c_{4} = c_{6} - c_{5}$.
\end{proposition}
\begin{proof}
	Differentiating $\theta(x,y)$ and replacing into (\ref{sineGordoneq}) we have
	\[
	(C'(x) + D'(y))\cos(\theta(x,y)) - (C^2(x) + D^2(y) - 4)\cos(\theta(x,y))\sin(\theta(x,y)) = 0.
	\]
	Using the fact that $\sin\theta(x,y) = \tanh(C(x) + D(y))$ and dividing with $\cos(\theta(x,y))$ we get
	\[
	\tanh(C(x) + D(y)) = \frac{C''(x) + D''(y)}{(C'(x))^2 + (D'(y))^2 - 4}.
	\]
	Differentiating with respect to x and using $c(x) = C'(x)$, $d(y) = D'(y)$ and the identity $sech^2(\phi) = 1 - \tanh^2(\phi)$ we obtain
	\begin{equation}\label{candd}
		c(x)\left((c^2(x) + d^2(y) - 4)^2 - (c'(x) + d'(y))^2\right) 
	\end{equation} 
\[
= c''(x)(c^2(x) + d^2(y) - 4) - 2c(x)c'(x)(c'(x) + d'(y)).
\]
Differentiating with respect to y we get
\[
4c(x)d(y)d'(y)(c^2(x) + d^2(y) - 4) - 2c(x)(c'(x) + d'(y))d''(y) 
\]
\[
= 2c''(x)d(y)d'(y) - 2c(x)c'(x)d''(y).
\]
Therefore we derive the equations
\[
\frac{c''(x)}{c(x)} - 2c^2(x) = - \frac{d''(y)}{d(y)} + 2d^2(y) - 8 = c_4.
\]
So we get the equations
\begin{equation}\label{c}
c''(x) - 2c^3(x) = c_4c(x),
\end{equation}
\begin{equation}\label{d}
d''(y) - 2d^3(y) = - (8 + c_4)d(y).
\end{equation}
Starting with equation (\ref{c}), we have
\[
2c'(x)c''(x) = 4c^3(x)c'(x) + 2c_4c(x)c'(x)
\]
and by integrating both parts we get (\ref{cdiffeq}). By the same arguements we obtain equation 
(\ref{ddiffeq}) and replacing them in (\ref{candd}) we get the relationship $16 + 4c_4 = c_6 - c_5$.
\end{proof}
\begin{corollary}\label{c0andd0}
	By using the initial conditions $c(0) = d(0) = 0$, we have $c_5 = 16\gamma^2$, $c_6 = 16\delta^2$ and $c_4 = 4(\gamma^2 - \delta^2 - 1)$ and the equations (\ref{cdiffeq}) and (\ref{ddiffeq}) turn into
	\begin{equation}\label{ceq}
		c'(x)^2 = c^4(x) + 4(1 + \delta^2 - \gamma^2)c^2(x) + 4\gamma^2,
	\end{equation}
\begin{equation}\label{deq}
	d'(y)^2 = d^4(y) + 4(1 + \gamma^2 - \delta^2)d^2(y) + 4\delta^2. 
\end{equation}
\end{corollary}
We shall demonstrate now how to apply the B{\"a}cklund transformation to derive the corresponding solution of the sinh-Gordon equation. To simplify the calculations we define a new variable system
\[
	X(x) = \int_{0}^{x} \sin\theta(t,0)dt, \quad Y(x,y) = \int_{0}^{y} \cos\theta(x,s)ds.
\]
	\begin{proposition}
	If $\theta(x,y)$ is a solution of (\ref{sineGordoneq}) of the form (\ref{sin(theta)}) and $d(0) = 0$ then 
	\begin{equation}\label{Sol3}
	\tanh(\frac{w(x,y)}{2}) = L 
	\frac{\tanh(\frac{w(0,0)}{2})e^{-2X} + L\tan(\frac{\sqrt{|4 - c^2(x)|}}{2}Y)}{L - \tanh(\frac{w(0,0)}{2})e^{-2X}\tan(\frac{\sqrt{|4 - c^2(x)|}}{2}Y)}
	\end{equation}
	where $L = L(x) = \frac{\sqrt{|4 - c(x)^{2}|}}{c(x) - 2}$.
\end{proposition}
\begin{proof}
	\par By setting $y = 0$ in (\ref{Back1}), we get
	\[
		\frac{\partial w(x,0)}{\partial x} = -2\sinh w(x,0)\sin\theta(x,0)
	\]
	\[
	\int_{w(0,0)}^{w(0,y)} \frac{ds}{\sinh s} = - 2\int_{0}^{x} \sin\theta(t,0)ds = -2X(x)
	\]

	\begin{align}
		\tanh(\frac{w(x,0)}{2}) = \tanh(\frac{w(0,0)}{2})\exp(-2X(x)).
	\end{align}
	\par To calculate $w(x,y)$ we use equation (\ref{Back2}) to obtain that
	\[
	\frac{\partial w(x,y)}{\partial y} = - (2\cosh w(x,y) + c(x))\cos\theta(x,y)
	\]
	\[
	\int_{w(x,0)}^{w(x,y)}  \frac{dt}{c(x) + 2\cosh t} = -\int_{0}^{y}\cos\theta(x,s)ds = -Y(x,y)
	\]
	\[
	\arctan(\frac{1}{L}\tanh(\frac{w(x,y)}{2}) - \arctan(\frac{1}{L}\tanh(\frac{w(x,0)}{2})) = 
	\frac{\sqrt{|4 - c^{2}|}}{2}Y
	\]
	where $L = L(x) = \frac{\sqrt{|4 - c(x)^{2}|}}{c(x) - 2}$ and so 
	\[
	\tanh(\frac{w(x,y)}{2}) = L 
	\tan(\arctan(\frac{\exp(-2X)}{L}) + \frac{\sqrt{|4 - c^{2}|}}{2}Y).
	\]
	Using the trigonemetric identity $\tan(x + y) = \frac{\tan(x) + \tan(y)}{1 - \tan(x)\tan(y)}$ we derive forlula (\ref{Sol3}).
\end{proof}
Since we calculated a family of solutions of the sinh-Gordon equation, we can construct the corresponding harmonic map with target the upper-half plane.
\begin{example}
	Using Corollary \ref{c0andd0} and choose the coefficients $\gamma = \delta = 1$. Then 
	$c(x) = \sqrt{2}\tanh(\sqrt{2}x)$ and $d(y) = - \sqrt{2}\tanh(\sqrt{2}y)$. We have to remark that we choose the sign of $c(x)$ and $d(y)$ so that 
	\[
	-1 < \frac{c'(x) + d'(y)}{c^2(x) + d^2(y) - 4} < 1.
	\]
	Then we get a solution $\theta$ of the sine-Gordon equation (\ref{sineGordoneq}) that is given by
	\[
	\tan(\frac{\theta(x,y)}{2}) = \frac{\cosh(\sqrt{2}x) - \cosh(\sqrt{2}y)}{\cosh(\sqrt{2}x) + \cosh(\sqrt{2}y)}.
	\]
	The corresponding solution $w$ of the sinh-Gordon equation (\ref{sinh-Gordon eqn}) is given by
	\[
	\tanh(\frac{w(x,y)}{2}) = \frac{\sqrt{2}\sh(\sqrt{2}y)}{\sqrt{2}\sh(\sqrt{2}x) - 2\ch(\sqrt{2}x)}.
	\]
	To construct a harmonic map we use Proposition \ref{PPFD} and with initial conditions $S(0,0) = 1/2$ and $R(0,0) = 0$ we get
	\[
	I_1(x) = \int_{0}^{x}\frac{\ch^2(\sqrt{2}t) - 1}{\ch^2(\sqrt{2}t) + 1}dt = x - \arctanh(\frac{\tanh(\sqrt{2}x)}{\sqrt{2}}),
	\]
	\[
	I_2(x,y) = \int_{0}^{y} \frac{4\sqrt{2}fg\sh(\sqrt{2}s)\ch(\sqrt{2}s)}
	{(f^2 - 2\sh^2(\sqrt{2}s))(g^2 + \ch^2(\sqrt{2}s))}ds 
	\]
	\[
	= \frac{1}{2}\log(\frac{f^2 + 1 - \ch(2\sqrt{2}y)}{2g^2 + \ch(2\sqrt{2}y) + 1}
	\frac{2g^2 + 2}{f^2}
	),
	\]
		\[
	I_3(x) = \int_{0}^{x} e^{2x}e^{-2\arctanh(\frac{\tanh(\sqrt{2}x)}{\sqrt{2}})}
	\frac{4\ch(\sqrt{2}t)}{3 + \ch(2\sqrt{2}t)}dt
	\]
	\[
	= 
	\frac{4e^{2x}}
	{4\ch(\sqrt{2}x) + 2\sqrt{2}\sh(\sqrt{2}x)} - 1
	\]
	\[
	I_4(x,y) = 
	\frac{2g^2 + 2}{f^2}
	\int_{0}^{y}
	\frac{\sh(\sqrt{2}s)(g^2 - \ch^2(\sqrt{2}s))}{(g^2 + \ch^2(\sqrt{2}s))^2}ds
	\]
	\[
	= 2e^{2x}
	(
	\frac{\ch(\sqrt{2}y)(\sqrt{2}\sh(\sqrt{2}x) - 2\ch(\sqrt{2}x))}{2 + \ch(2\sqrt{2}x) + \ch(2\sqrt{2}y)}
	- \frac
	{\sqrt{2}\sh(\sqrt{2}x) - 2\ch(\sqrt{2}x)}
	{3 + \ch(2\sqrt{2}x)}
	)
	\]
	where $f = f(x) = \sqrt{2}\sh(\sqrt{2}x) - 2\ch(\sqrt{2}x)$ and $g = g(x) = \ch(\sqrt{2}x)$. Let
	\begin{align}
		S(x,y) = e^{2x}
		\frac{2 + 3\ch(2\sqrt{2}x) - \ch(2\sqrt{2}y) - 2\sqrt{2}\sh(2\sqrt{2}x)}
		{2 + \ch(2\sqrt{2}x) + \ch(2\sqrt{2}y)}.
	\end{align}
	\begin{align}
		R(x,y) =4e^{2x}
		\left(\frac{\ch(\sqrt{2}y)(2\ch(\sqrt{2}x) - \sqrt{2}\sh(\sqrt{2}x))}{2 + \ch(2\sqrt{2}x) + \ch(2\sqrt{2}y)}\right)
	- 2
	\end{align}

Then $u = R + iS$ a harmonic map from a surface to the upper-half plane equipped with Poicare metric.
\end{example}

\textbf{Conclusions.}
We have found new solutions of the sinh-Gordon equation and the sine-Gordon equation, motivated by the B{\"a}cklund transformation which relates solutions of these equations. Next in each such class of solutions we have provided an example of a corresponding harmonic map. Our aim next is to study sine-Gordon type equations, not necessarily elliptic, and their symmetries in order to extend our results. This project is currently under investigation and it is motivated by the study of wave maps between Lorentz spaces.

\textbf{Acknowledgements.} The author would like to thank Prof. C. Daskaloyannis, Prof. A. Fotiadis and Dr. E. Papageorgiou for their help and support.

\vspace{20pt}
\address{
	\noindent\textsc{Giannis Polychrou:}
	\href{mailto:ipolychr@math.auth.gr}
	{ipolychr@math.auth.gr}\\
	Department of Mathematics, Aristotle University of Thessaloniki,
	Thessaloniki 54124, Greece}
\end{document}